\documentclass{article}
\usepackage[utf8]{inputenc}
\usepackage{amsmath}
\usepackage{amssymb}
\usepackage{amsthm}
\usepackage{verbatim}
\usepackage{graphicx}
\usepackage{subcaption}
\usepackage[shortlabels]{enumitem}
\usepackage{xcolor}
\usepackage{float}
\usepackage{lineno}
\usepackage{cases}
\usepackage{bm}

\newtheorem{theorem}{Theorem}[section]
\newtheorem{corollary}{Corollary}[theorem]

\newtheorem{proposition}[theorem]{Proposition}

\newcommand{\KK}{{\mathsf{K_h}}}
\newcommand{\AAA}{{\mathsf{A}}}

\newcommand{\AAd}{{\mathsf{A}^\dagger}}
\newcommand{\NAT}{{\mathcal{N}(\AAA)^\bot}}

\newcommand{\JJ}{\mathcal{J}}

\newcommand{\WW}{{\mathsf{W}}}

\newcommand{\PP}{\mathsf{P}}

\newcommand{\II}{{\mathsf{I}}}
\newcommand{\xx}{{\mathbf{x}}}
\newcommand{\yy}{{\mathbf{y}}}
\newcommand{\zz}{{\mathbf{z}}}
\newcommand{\vv}{{\mathbf{v}}}

\newcommand{\bb}{{\mathbf{b}}}
\newcommand{\cc}{{\mathbf{c}}}
\newcommand{\dd}{{\mathbf{d}}}
\newcommand{\ttau}{{\bm{\tau}}}
\newcommand{\qq}{{\mathbf{q}}}
\newcommand{\ee}{{\mathbf{e}}}

\newcommand{\nn}{{\mathbf{n}}}

\newcommand{\Rn}{{\mathbb{R}^n}}

\newcommand{\Rnn}{{\mathbb{R}^{n\times n}}}
\newcommand{\Rmn}{{\mathbb{R}^{m\times n}}}

\newcommand{\nullspace}[1]{{\mathcal{N}(#1)}}

\DeclareMathOperator*{\argmin}{arg\,min}
\DeclareMathOperator*{\argmax}{arg\,max}

\title{Box constraints and weighted sparsity regularization for identifying sources in elliptic PDEs}

\author{Ole L{\o}seth Elvetun\thanks{Faculty of Science and Technology, Norwegian University of Life Sciences, P.O. Box 5003, NO-1432 {\AA}s, Norway. Email: ole.elvetun@nmbu.no.} and Bj{\o}rn Fredrik Nielsen\thanks{Faculty of Science and Technology, Norwegian University of Life Sciences, P.O. Box 5003, NO-1432 {\AA}s, Norway. Email: bjorn.f.nielsen@nmbu.no.}}

\begin{document}

\maketitle

\begin{abstract}
  We explore the possibility for using boundary data to identify sources in elliptic PDEs. Even though the associated forward operator has a large null space, it turns out that box constraints, combined with weighted sparsity regularization, can enable rather accurate recovery of sources with constant magnitude/strength. In addition, for sources with varying strength, the support of the inverse solution 
  will be a subset of the support of the true source.  We present both an analysis of the problem and a series of numerical experiments. Our work only addresses discretized 
  problems. 
  
  The reason for introducing the weighting procedure is that standard (unweighted) sparsity regularization fails to provide adequate results for the source identification task considered in this paper. 
  This investigation is also motivated by applications, e.g., recovering mass distributions from measurements of  gravitational fields and 
  inverse scattering. We develop the methodology and the analysis in terms of Euclidean spaces, and our results can therefore be applied to many problems. For example, the results 
  are equally applicable to models involving the screened Poisson equation as to models using the  Helmholtz equation, with both large and small wave numbers. 
\end{abstract}


\section{Introduction}
We study the task of identifying the source $f$ in an elliptic PDE from boundary data: 
    \begin{equation}\label{eq1}
        \min_{(f,u) \in F_h \times H^1(\Omega), \, 0 \leq f \leq s} \left\{ \frac{1}{2}\|u - d\|_{L^2(\partial\Omega)}^2 + \alpha \sum_i w_i |(f, \phi_i)_{L^2(\Omega)}| \right\}
    \end{equation}
    subject to 
    \begin{equation}\label{eq2}
    \begin{split}
        -\Delta u + \epsilon u &= f \quad \mbox{in } \Omega, \\
        \frac{\partial u}{\partial \nn} &= 0  \quad \mbox{on } \partial \Omega, 
    \end{split}
    \end{equation} 
where $\Omega \subset \mathbb{R}^\nu$ is a Lipschitz domain with boundary $\partial \Omega$, $\{ \phi_1, \phi_2, \ldots, \phi_n \}$ is an $L^2$-orthonormal basis for the finite dimensional function space $F_h$, $\alpha > 0$ is a regularization parameter, $d$ represents Dirichlet boundary data and $\nu=1,2$ or $3$. In this paper we use the same positive weights $\{ w_i \}_{i=1}^n$ in the $\ell^1$-regularization term in \eqref{eq1} as in \cite{Elv21,Elv21b}. Their precise form is presented in the next section. The upper bound $s$ in \eqref{eq1} for the strength of the source function $f$ will be either finite or infinite, i.e., $s \in (0, \infty]$, and we allow the constant $\epsilon$ to be either positive or negative and of arbitrary fixed size: $$\epsilon = -\kappa^2 \mbox{ or } \epsilon = \kappa^2.$$ This means that our results hold for both the screened Poisson equation and the Helmholtz equation. 

Regarding the basis $\{ \phi_1, \phi_2, \ldots, \phi_n \}$ for the source space $F_h$, we mention that one can, e.g., for $i=1,2,\ldots,n$, choose $\phi_i$ to be the normalized characteristic function of the $i$'th cell/subdomain in a mesh for $\Omega$. Provided that the measure of the overlaps between different cells is zero, this yields an $L^2$-orthonormal basis.     

Variants of \eqref{eq1}-\eqref{eq2} occur in a number of applications. For example, in connection with  inverse scattering problems \cite[section 1.7]{BEng96}. It has therefore been studied by many mathematicians, and a number of methods for computing reliable results have been published. 

Often one assumes that $f$ consists of a finite number of point sources or sources having small compact support, see, e.g., \cite{Abdel15,babda09,Bad00,Han11,Zha19}. This leads to involved mathematical issues, but sometimes optimization procedures or explicit regularization can be avoided. Also, such "direct methods" can potentially be used to identify more general sources \cite{Wan17,Zha18} when $\epsilon < 0$, i.e.,  for the Helmholtz equation with multi-frequency data.   

In \cite{cheng15,song12} the authors suggest to employ a Kohn-Vogelius fidelity term and a restricted control domain to identify the source. A similar approach is considered in \cite{hinze19}, but instead of limiting the control domain, a suitable prior is used.  

Assuming that $f$ has the form $f(x)=\rho_1, \, \rho_2$ in $\Omega_1, \, \Omega_2$, respectively, where $\rho_1$ and $\rho_2$ are given constants, the authors of \cite{kun94, ring95} explored the possibility of  identify the subdomains $\Omega_1$ and $\Omega_2$. These investigations are motivated by, e.g., the task of determining the mass distribution, i.e. the core-mantle boundary, of the earth from measurements of the gravitational field. 

If the boundary potential associated with two constant strength sources with  star-shaped supports\footnote{Star-shaped with respect to their centers of gravity.} coincide, then these two sources must be identical, see \cite[Theorem 4.1.1]{BIsa05} which contains an analysis of this when $\epsilon =0$ (Poisson's equation).  
%
%
In \cite{Ike99} a more general case is considered for the Helmholtz equation, and the author of that paper shows that it is possible to recover the {\em support function}
of a polygonal shaped support, see also \cite{Ikehata_2021}.  Furthermore, during the last three decades very advanced investigations have been presented to further clarify the possibilities for computing the support of sources (or scattering objects) from boundary (or far field) data, see, e.g., \cite{Griesmaier_2012,Gri13,Hanke_2012,Hanke_2008,Han11,Het96,Kress_2013,Kusiak_2003,Kusiak_2005,Potthast_2003}. Several of these papers address the important task of computing the so-called {\em convex source (or scattering)  support}. 

It turns out that the use of box constraints and $\ell^1$-regularization in \eqref{eq1}-\eqref{eq2} enable the recovery of several point/small sources as well as the identification of rather large composite sources with constant strength, provided that some fairly mild assumptions hold. Both of these cases can be handled with the same methodology. We also present results which make it possible to estimate the constant strength of $f$, as well as the support of sources with varying strength. Since software for sparsity regularization is readily available, our scheme can be implemented in a rather straightforward manner using standard PDE solvers.  

Our approach is developed in terms of Euclidean spaces, and it can therefore be applied to a rather large class of problems. This, e.g., implies that we do not need to impose limiting restrictions on the domain $\Omega$ and that the methodology can handle state equations involving  conductivity functions or tensors, i.e., that the term $-\Delta u$ in \eqref{eq2} instead has the form $-\nabla \cdot (\sigma \nabla u)$. 

The mathematical analysis is presented in the next section: Both the basis pursuit version and the regularized form of the problem are explored. Section \ref{sec:numerical_experiments} contains the numerical experiments, including tests with noisy data, and we close the paper with some remarks in section \ref{sec:concluding_remarks}. 

\section{Analysis}
Let 
\begin{equation*} 
    K_h:f \mapsto u|_{\partial \Omega}
\end{equation*}
be the forward operator associated with \eqref{eq1}-\eqref{eq2}, where $u=u(f)$ denotes the solution of the state equations \eqref{eq2} for a given $f$. Then we may express our source identification problem in the form 
\begin{equation} \label{eq3}
    \min_{f\in F_h, \, 0 \leq f \leq s} \left\{ \frac{1}{2}\|K_h f - d\|_{L^2(\partial\Omega)}^2 + \alpha \sum_i w_i |(f, \phi_i)_{L^2(\Omega)}| \right\}. 
\end{equation}

Upon discretization of the boundary value problem \eqref{eq2}, we get the transfer/forward matrix $\AAA \in \Rmn$, and we obtain the following fully discrete version of \eqref{eq3} 
\begin{equation} \label{eq:3.1}
    \min_{0 \leq \xx \leq s} \left\{ \frac{1}{2}\|\AAA \xx - \bb\|_2^2 + \alpha \sum_i w_i |x_i| \right\}. 
\end{equation}
Here, $\xx$ represents the Euclidean vector of $f$ relative to the $L^2$-orthonormal basis $\{ \phi_1, \phi_2, \ldots, \phi_n \}$ of $F_h$, i.e., 
\begin{equation} \label{eq:5.1}
    f = \sum_i x_i \phi_i.
\end{equation}
As mentioned above, for the problem \eqref{eq1}-\eqref{eq2} one can, e.g., choose 
\begin{equation} \label{eq:5.2}
    \phi_i = \| \chi_{\Omega_i} \|_{L^2(\Omega)}^{-1} \, \chi_{\Omega_i}, 
\end{equation}
where $\chi_{\Omega_i}$ denotes the characteristic function of the $i$'th cell $\Omega_i$ in a mesh for $\Omega$. 

Let 
\begin{equation}\label{eq:5.21} 
\WW = \mbox{diag}(w_1,w_2,\ldots,w_n) 
\end{equation}
denote the diagonal matrix with diagonal entries $\{ w_i \}$. Then the associated basis pursuit (zero regularization limit) problem to \eqref{eq:3.1} reads 
\begin{equation} \label{eq5}
      \min_{0 \leq \xx \leq s} \| \WW \xx \|_1 \quad \textnormal{subject to} \quad \AAA \xx = \bb. 
\end{equation}

Since only boundary data is available for the recovery of $f$, the operator $K_h$ and the matrix $\AAA$ will have nontrivial null spaces. This differs significantly from the typical setup for optimal control problems, which often assumes observations of the state $u$ throughout the entire domain $\Omega$. Nevertheless, the nature of many important inverse problems is such that only very limited observation data is available.   

Our work is motivated by the inverse problems mentioned above. But for the sake of generality, we will develop the methodology and analysis in terms of Euclidean spaces. It should also be mentioned that it remains an open problem how to establish an infinite dimensional counterpart to the results presented in this paper. 

\subsection{Diagonal weight/regularization matrix $\WW$} \label{subsec:weight_matrix}
Assume that $\AAA \in \Rmn$ has a nontrivial null space $\nullspace{\AAA}$ and consider the equation
\begin{equation}\label{eq:1}
    \AAA\xx = \bb.
\end{equation}
We can multiply \eqref{eq:1} with the pseudo-inverse $\AAd$ of $\AAA$ to obtain
\begin{equation}\label{eq:2}
 \AAd\AAA\xx = \AAd\bb,
\end{equation}
where we recall that $\AAd\AAA$ is the orthogonal projection
\begin{equation}\label{eq:3}
    \PP = \AAd\AAA : \Rn \rightarrow \NAT. 
\end{equation}

The authors of \cite{Elv21} introduced the diagonal weight/regularization matrix $\WW \in \Rnn$ defined by\footnote{In section 3 in \cite{Elv21}, $\WW$ is defined in terms of two abstract finite dimensional inner product space $X$ and $Y$ and a linear operator $\KK: X \rightarrow Y$. In this paper we use the "Euclidean version" of this.} 
\begin{equation}\label{eq:4}
    \WW \ee_i := \underbrace{\|\PP\ee_i\|_2}_{:=w_i}\ee_i, \quad i=1,2,\ldots, n, 
\end{equation}
where $\ee_1, \ee_2, \ldots, \ee_n$ are the standard unit basis vectors. Throughout this paper we assume that none of the standard unit basis vectors belong to the null space $\nullspace{\AAA}$ of $\AAA$, which implies that 
\begin{equation}
    w_i=\|\PP\ee_i\|_2 \neq 0 , \quad i= 1,2,\ldots, n, 
\end{equation}
and $\WW$ becomes invertible since it is diagonal \eqref{eq:5.21}. 

Some of the mathematical properties of $\WW$ are explored in \cite{Elv21,Elv21b}. In order to motivate the use of this type of weighting procedure, 
we mention the following: Let $\xx^\dagger$ be the minimum $\| \cdot \|_2$-norm solution of 
\begin{equation} \label{eq:4.1}
    \AAA \xx = \AAA \ee_j. 
\end{equation}
Then, see Theorem 2.1 in \cite{Elv21b} and Theorem 4.2 in \cite{Elv21}, 
 \begin{align} \nonumber 
     j &= \argmax_{i \in \{1,2,...,n\}} |[\WW^{-1}\xx^\dagger]_i| \\
     \label{eq:max_property}
     &= \argmax_{i \in \{1,2,...,n\}} |[\WW^{-1}\AAA^\dagger \AAA \ee_j]_i|, 
 \end{align}
 where $[\mathbf{v}]_i$ represents the $i$'th component of the Euclidean vector $\mathbf{v}$. This means that $\WW^{-1}\xx^\dagger$ can be used to recover the correct index $j$ from the "data" $\bb^\dagger = \AAA \ee_j$ in \eqref{eq:4.1}.  
 


\subsubsection*{Remarks} 
In order to obtain a "sound" regularization operator $\WW$, as defined in \eqref{eq:4}, one should ensure that $\min_i \{ \| \PP \ee_i \|_2 \}$ is not too small. This follows from the observation that, if $\| \PP \ee_i \|_2 \approx 0$, then one tries to recover the contribution associated with a basis vector $\ee_i$ which is almost in the null space of the matrix $\AAA$. Thus, in order to avoid to get "very close" to the null space, one would typically use a rather moderate number $n$ of basis functions $\phi_1, \phi_2, \ldots, \phi_n$ to discretize the source $f$ in \eqref{eq1}-\eqref{eq2}, 
see \eqref{eq:5.1} and \eqref{eq:5.2}. This is our motivation for employing a finite dimensional space for $f$ in \eqref{eq1}-\eqref{eq2}.

The computation of $\PP \ee_i = \AAA^\dagger \AAA \ee_i$ involves the pseudoinverse/Moore-Penrose inverse $\AAA^\dagger$ of $\AAA$. If $\AAA$ has small nonzero singular values, which will typically be the case for a discretized ill posed problem, then it is not advisable to apply $\AAA^\dagger$ in computations, and we must "replace" $\AAA^\dagger$ by a more well-behaved matrix. This can, e.g., be accomplished by employing truncated SVD or by applying Tikhonov regularization, with regularization parameter $\gamma$, 
\begin{equation*}
    (\AAA^T \AAA  + \gamma \II) \xx_\gamma = \AAA^T \AAA \ee_i. 
\end{equation*}
Then, according to standard theory, $\lim_{\gamma \rightarrow 0^+} \xx_\gamma = \AAA^\dagger \AAA \ee_i = \PP \ee_i$. In our numerical experiments (section \ref{sec:numerical_experiments}) we used truncated SVD 
to compute an approximation of the entries in the diagonal regularization matrix $\WW$. 

\subsection{Basis pursuit} 
We will now investigate whether one can use\footnote{If $\AAA$ has very small nonzero singular values, then it is in practise not possible to solve \eqref{eq5}, and one must instead consider a regularized version of the problem. The latter is discussed in the next subsection.} \eqref{eq5} to approximately, or partially, recover a true source 
\begin{equation*}
    \xx^* = \sum_{\JJ} x_j^* \ee_j, \quad \xx^* \geq 0,
\end{equation*}
from the synthetic data $\bb=\bb^\dagger= \AAA \xx^*$, where $\JJ$ is the support of $\xx^*$.

Our first result shows that requiring positivity may enable the identification of the support of the sources.
\begin{theorem} \label{lower_constraint}
Let $\JJ = \mathrm{supp} (\xx^*)$ 
and assume that there exists a vector $\cc$ such that 
\begin{align}
\label{IC1}
\frac{\PP \ee_j}{\| \PP \ee_j \|_2} \cdot \cc &= 1 \quad \forall j \in \JJ, \\
\label{IC2}
\frac{\PP \ee_i}{\| \PP \ee_i \|_2} \cdot \cc &< 1 \quad \forall i \in \JJ^c. 
\end{align}
Then any solution $\yy$ of 
\begin{equation}
\label{IC3}
      \min_{0 \leq \xx \leq s} \| \WW \xx \|_1 \quad \textnormal{subject to} \quad \AAA \xx = \AAA \overbrace{\sum_{\JJ} x_j^* \ee_j}^{= \xx^*}
\end{equation}
satisfies 
\[
\mathrm{supp} (\yy) \subseteq \mathrm{supp} (\xx^*), 
\]
provided that $\xx^* \geq 0$ and that $s \in [\| \xx^* \|_\infty, \infty]$. 
Also, $\xx^*$ is a solution of \eqref{IC3}. 
\end{theorem}
\subsubsection*{Remark} 
Our geometric intuition for assumptions \eqref{IC1} and \eqref{IC2} is: There exists a direction $\cc$ such that $\PP \ee_j / \| \PP \ee_j \|_2$, $\forall j \in \JJ$, and $\PP \ee_i / \| \PP \ee_i \|_2$, $\forall i \in \JJ^c$, are "equally important" and "less important", respectively, in this direction. We will in subsection \ref{sec:optimality_conditions} explain how \eqref{IC1} and \eqref{IC2} can be derived from the classical KKT-conditions. 
\begin{proof}[Proof of Theorem \ref{lower_constraint}]
Let $$\yy = \sum_{\JJ^c \cup \JJ} y_i \ee_i, \quad y_i \geq 0,$$ denote a solution of \eqref{IC3}. Then 
\[
\AAA \yy = \AAA \sum_{\JJ} x_j^* \ee_j
\]
or, because $\PP = \AAA^\dagger \AAA$,  
\[
\PP \yy = \PP \sum_{\JJ} x_j^* \ee_j. 
\]
Taking the inner product with $\cc$ yields 
\begin{equation*}
\sum_{\JJ^c \cup \JJ} y_i \PP \ee_i \cdot \cc  
= \sum_{\JJ} x_j^* \PP \ee_j \cdot \cc 
\end{equation*}
which can be written in the form, using \eqref{IC1},  
\begin{equation} \label{IC4} 
\sum_{\JJ^c \cup \JJ} y_i \| \PP  \ee_i \|_2 \frac{\PP \ee_i}{\| \PP  \ee_i \|_2} \cdot \cc  
= \sum_{\JJ} x_j^* \| \PP  \ee_j \|_2 .
\end{equation}

Assume now that there is $i' \in \JJ^c$ such that $0 < y_{i'} \leq s$. Then, combining \eqref{IC4} with \eqref{IC1}, \eqref{IC2} and the inequality constraint $y_i \geq 0$ lead to the conclusion that 
\begin{equation} \label{IC4.1}
\sum_{\JJ^c \cup \JJ} y_i \| \PP  \ee_i \|_2 > \sum_{\JJ} x_j^* \| \PP  \ee_j \|_2. 
\end{equation}
That is, recall the definition \eqref{eq:4} of $\WW$,  
\begin{align}
\nonumber
 \| \WW \yy \|_1 &= \| \sum_{\JJ^c \cup \JJ} y_i \WW \ee_i \|_1 \\ 
 \nonumber
 &=\sum_{\JJ^c \cup \JJ} y_i \| \PP  \ee_i \|_2 \\
 \label{IC5}
 &> \sum_{\JJ} x_j^* \| \PP  \ee_j \|_2 \\
 \nonumber
 &= \| \WW \xx^* \|_1 ,
\end{align}
where we have used the assumption that $\xx^* \geq 0$ in the last equality. Consequently, since $\xx^*$ satisfies the constraints in \eqref{IC3}, there can not be any $i' \in \JJ^c$ such that $y_{i'}>0$. We conclude that $\mathrm{supp} (\yy) \subseteq \JJ$.  

When $\mathrm{supp} (\yy) \subseteq \JJ$, we must replace $>$ with $\geq$ in \eqref{IC4.1} and \eqref{IC5}. It follows that 
\[
\| \WW \yy \|_1 \geq \| \WW \xx^* \|_1
\]
for any solution $\yy$ of \eqref{IC3}, and we conclude that $\xx^*$ solves \eqref{IC3}. 
\end{proof}

As an immediate consequence we obtain the following result, where we remark that: Gradually decreasing $s$ in \eqref{IC3} may enforce uniqueness because the size of $\left\{ \zz \, | \, \mathrm{supp} (\zz) \subseteq \JJ \mbox{ and } 0 \leq \zz \leq s \right\}$ will decrease as $s$ decreases. 
\begin{corollary} \label{cor:injective}
Assume that $\xx^* = \sum_{\JJ} x_j^* \ee_j$, $\xx^* \geq 0$, is such that a vector $\cc$ satisfying \eqref{IC1} and \eqref{IC2} exists. If $\AAA$ is injective on $\left\{ \zz \, | \, \mathrm{supp} (\zz) \subseteq \JJ \mbox{ and } 0 \leq \zz \leq s \right\}$, where $s \in \left[ \| \xx^* \|_\infty, \infty \right]$, then $\xx^*$ is the unique solution of \eqref{IC3}. 
\end{corollary}

Recall the definition of the "$\ell^0$-norm"  
$$
\| \vv \|_0 = \# \{i \, | \, v_i \neq 0 \}.
$$ 
Roughly speaking,  the next result explains that, if a minimum "$\ell^0$-norm" solution with support $\JJ$ obeying \eqref{IC1} and \eqref{IC2} exists, then all solutions of the basis pursuit problem must have the same support -- i.e., the support is unique.
\begin{corollary}
Assume that 
\begin{equation*}
    \bb \in \{ \AAA \zz \, | \, \zz\geq 0 \}, 
\end{equation*}
and that there exist
\begin{equation} \label{IC5.001}
    \widetilde{\xx}^* \in \argmin_{\xx \geq 0} \| \xx \|_0 \quad \textnormal{subject to} \quad \AAA \xx = \bb  
\end{equation}
and a vector $\cc$ such that \eqref{IC1} and \eqref{IC2} hold with $\JJ = \textnormal{supp}(\widetilde{\xx}^*)$. 
Then any solution $\yy$ of 
\begin{equation} \label{IC5.01}
    \min_{0 \leq \xx \leq s} \| \WW \xx \|_1 \quad \textnormal{subject to} \quad \AAA \xx = \bb
\end{equation}
also solves the problem 
\begin{equation*}
    \min_{0 \leq \xx \leq s} \| \xx \|_0 \quad \textnormal{subject to} \quad \AAA \xx = \bb 
\end{equation*}
and 
\begin{equation*}
    \mathrm{supp} (\yy) = \mathrm{supp} (\widetilde{\xx}^*).
\end{equation*}
This holds for any $s \in \left[ \| \widetilde{\xx}^* \|_\infty, \infty \right]$.
\end{corollary}
\begin{proof}
Since $\AAA \widetilde{\xx}^* = \bb$, any solution $\yy$ of \eqref{IC5.01} must also be solution of \eqref{IC3} with $\xx^*=\widetilde{\xx}^*$. Therefore, according to Theorem \ref{lower_constraint}, $\yy$ satisfies \begin{equation} \label{IC5.02} 
\mathrm{supp} (\yy) \subseteq \mathrm{supp} (\widetilde{\xx}^*).
\end{equation} 
Consequently, 
\[
\| \yy \|_0 \leq \| \widetilde{\xx}^* \|_0, 
\]
but from the definition \eqref{IC5.001} of $\widetilde{\xx}^*$  
\[
\| \yy \|_0 \geq \| \widetilde{\xx}^* \|_0, 
\]
and we conclude that $\| \yy \|_0 = \| \widetilde{\xx}^* \|_0$, which together with \eqref{IC5.02} complete the argument.
\end{proof}

We now show that precise information about the magnitude/strength may make it possible to (perfectly) recover sources with constant strength 
\begin{equation} \label{IC5.1}
    \xx^* = k \sum_{\JJ} \ee_j, \quad k >0,  
\end{equation}
i.e., $x_j^* = k$ for all $j \in \JJ=\mathrm{supp} (\xx^*)$. 
\begin{corollary} 
\label{lower_upper_constraint}
If there exists a vector $\cc$ such that \eqref{IC1} and \eqref{IC2} hold, then $\xx^*$
is the unique solution of 
\begin{equation}
\label{IC6}
      \min_{0 \leq \xx \leq k} \| \WW \xx \|_1 \quad \textnormal{subject to} \quad \AAA \xx = \AAA \overbrace{k \sum_{\JJ} \ee_j}^{= \xx^*}. 
\end{equation}
\end{corollary}
\begin{proof}
Problem \eqref{IC6} is the same as problem \eqref{IC3} with a true source $\xx^*$ in constant strength form \eqref{IC5.1} and with $s=k$. Theorem \ref{lower_constraint} asserts that $\xx^*$ solves \eqref{IC6} and that any other solution $\yy$ of this problem must be such that $\mathrm{supp} (\yy) \subseteq \mathrm{supp} (\xx^*)$. Since both $\xx^*$ and $\yy$ solve \eqref{IC6}, it follows that $\| \WW \yy \|_1 = \| \WW \xx^* \|_1$, and we find that, see \eqref{eq:4},  
\begin{equation*}
    \sum_{\JJ} y_j \| \PP \ee_j \|_2 =  \sum_{\JJ} x_j^* \| \PP \ee_j \|_2 =  \sum_{\JJ}k \| \PP \ee_j \|_2 
\end{equation*}
or 
\begin{equation*}
    \sum_{\JJ} (k-y_j) \| \PP \ee_j \|_2 =  0.
\end{equation*}
This equality and the constraint $0 \leq y_j \leq k$, $\forall j \in \JJ$, lead to the conclusion that $y_j=k$ for all $j \in \JJ$ and hence that $\yy=\xx^*$. 
\end{proof}


In practice one typically does {\em not} have exact knowledge about the constant strength $k$ of the source(s). We now present a result which enables the recovery, under ideal conditions, of $k$. In what follows, we assume that $k$ is a fixed unknown positive number, and that we want to recover a constant strength source \eqref{IC5.1}. 

First we note that the state equation must have a solution in order for the basis pursuit problem \eqref{eq5} to be meaningful, and we thus introduce the set 
\begin{equation} \label{defS}
    S = \left\{ s \; | \; \exists \xx \mbox{ such that } \| \xx \|_{\infty} \leq s \mbox{ and } \quad \AAA \xx = \AAA \, k \sum_{\JJ} \ee_j \right\}. 
\end{equation}
\begin{corollary} 
\label{strength}
Assume that there exists a vector $\cc$ such that \eqref{IC1} and \eqref{IC2} hold for $\xx^*=k \sum_{\JJ} \ee_j$, $\JJ = \textnormal{supp}(\xx^*)$. Then the function $g:S \rightarrow \mathbb{R}$, 
\begin{equation*} 
    g(s):=\min_{0 \leq \xx \leq s} \| \WW \xx \|_1 \quad \textnormal{subject to} \quad \AAA \xx = \AAA \overbrace{k \sum_{\JJ} \ee_j}^{=\xx^*}, 
\end{equation*}
is non-increasing wrt. $s$ and satisfies 
\begin{align}
    \label{IC6.3}
    g(s) &> \| \WW \xx^* \|_1 \quad \mbox{for } s \in [0, k) \cap S, \\
    \label{IC6.2}
    g(s) &= \| \WW \xx^* \|_1 \quad \mbox{for } s \geq k. 
\end{align}
\end{corollary}
\begin{proof} 
The function $g(s)$ is non-increasing because the size of the constraining set $\{ \xx \, | \, 0 \leq \xx \leq s \}$ increases as $s$ increases. 

Since $\| \xx^* \|_\infty = k$, it follows from Theorem \ref{lower_constraint} that $\xx^*$ solves  
\begin{equation}
    \label{IC6.1}
    \min_{0 \leq \xx \leq s} \| \WW \xx \|_1 \quad \textnormal{subject to} \quad \AAA \xx = \AAA \xx^*
\end{equation}
when $s \geq k$. Consequently, \eqref{IC6.2} must hold. 

According to Corollary \ref{lower_upper_constraint}, $\xx^*$ is the unique solution of \eqref{IC6.1} when $s=k$. Hence, it follows that \eqref{IC6.3} must be true.  
\end{proof}

If the true source(s) has constant strength, we thus obtain the following scheme for recovering both its size and magnitude $k$: 
\subsubsection*{Algorithm} 
\begin{itemize}
    \item Compute 
    \begin{equation} \label{IC6.4}
      g(s) = \min_{0 \leq \xx \leq s} \| \WW \xx \|_1 \quad \textnormal{subject to} \quad \AAA \xx = \bb 
\end{equation}
many times, varying size of the upper bound $s$. 
\item The strength $k$ is then the number such that $g(s) = c$ (a constant) for $s \geq k$ and $g(s) > c$ for $s<k$.
\item Corollary \ref{lower_upper_constraint} assures the perfect recovery of the true constant strength source by solving \eqref{IC6.4} with $s=k$.
\end{itemize}
We will explore a regularized version of this algorithm in the numerical experiments section below. 

\subsubsection*{Remarks about assumptions \eqref{IC1} and \eqref{IC2}} 
Our results rely on whether there exists a vector $\cc$ such that \eqref{IC1} and \eqref{IC2} hold. This is a subtle question, which we will now discuss in some detail.  

For example, if the cardinality of $\JJ$ equals two, i.e., $\JJ = \{ j_1, j_2 \}$, then \eqref{IC1} holds with 
\[
\cc = \left( 1+ \frac{\PP \ee_{j_1}}{\| \PP \ee_{j_1} \|_2} \cdot \frac{\PP \ee_{j_2}}{\| \PP \ee_{j_2} \|_2} \right)^{-1} \left( \frac{\PP \ee_{j_1}}{\| \PP \ee_{j_1} \|_2} + \frac{\PP \ee_{j_2}}{\| \PP \ee_{j_2} \|_2} \right). 
\]
On the other hand, \eqref{IC2} may fail to hold for this particular choice of $\cc$. 

More generally, using the notation $\cc = \sum_{l=1}^n c_l \ee_l$, \eqref{IC1} leads to a linear system with $n$ unknowns $c_1, \, c_2, \ldots, \, c_n$, 
\begin{equation} \label{IC9}
    \sum_{l=1}^n  \ee_l\cdot \frac{\PP \ee_j}{\| \PP \ee_j \|_2} \, c_l - 1 = 0 \quad \forall j \in \JJ, 
\end{equation}
which must have a solution obeying the inequality constraints, see \eqref{IC2},  
\begin{equation} \label{IC10}
\begin{split}
    & \sum_{l=1}^n  \ee_l\cdot \frac{\PP \ee_i}{\| \PP \ee_i \|_2} \, c_l - 1 < 0 \quad \forall i \in \JJ^c.
\end{split}
\end{equation}
In other words, the solution set of \eqref{IC9} must overlap with the convex domain specified in \eqref{IC10}. 
This issue can be explored with standard software tools: We suggest to solve the following linear programming problem (LPP), where $\delta > 0$ is small fixed number, 
\begin{equation} \label{IC10.1}
    \max \sum_{i \in \JJ^c} \epsilon_i
\end{equation}
subject to $\{ c_l \}$ and $\{ \epsilon_i \}$ satisfying 
\begin{equation} \label{IC10.2}
\begin{split}
    & \sum_{l=1}^n  \ee_l\cdot \frac{\PP \ee_j}{\| \PP \ee_j \|_2} \, c_l - 1 = 0 \quad \forall j \in \JJ, \\
    & \sum_{l=1}^n  \ee_l\cdot \frac{\PP \ee_i}{\| \PP \ee_i \|_2} \, c_l - 1 + \delta + \epsilon_i = 0 \quad \forall i \in \JJ^c, \\
    & \epsilon_i \geq 0 \quad \forall i \in \JJ^c, 
\end{split}
\end{equation}
where $\{\epsilon_i \, | \, i \in \JJ^c\}$ are slack variables. If one during the solution process of \eqref{IC10.1}-\eqref{IC10.2} detects $\{ c_l \}$ and $\{ \epsilon_i \}$ which satisfy the constraints \eqref{IC10.2}, then one can stop the LPP program and conclude that the solution set of \eqref{IC9} overlaps with the convex domain described in \eqref{IC10}. That is, there exists a vector $\cc$ such that \eqref{IC1} and \eqref{IC2} hold. We also note that this approach only yields a sufficient criteria for asserting that \eqref{IC1} and \eqref{IC2} are satisfied.  

Let us also comment that, if the cardinality of the index set $\JJ$ is significantly smaller than the number $n$ of columns in the matrix $\AAA \in \mathbb{R}^{m \times n}$, then the solution set of \eqref{IC9} will typically be relatively large and it becomes more "likely" that there exists $\cc$ such that \eqref{IC10} holds. Similarly, if the size of $\JJ$ increases, then one would expect that the solution set of \eqref{IC9} decreases and it becomes less "likely" that \eqref{IC10} is satisfied, i.e., perfect source recovery with \eqref{IC3} or \eqref{IC6} will "probably" not be possible. (Here we use "likely" and "probably" in an informal manner.) 

In the event that there does {\em not} exist $\cc$ 
such that \eqref{IC9} and \eqref{IC10} hold, one might wonder whether it is possible to bound the size of the support of a solution $\yy$ of \eqref{IC3}. That is, to what extent will the size of the support of the true source $\xx^* = \sum_\JJ x_j^* \ee_j$ potentially be increased by solving \eqref{IC3}? 

To investigate this issue, assume that $\JJ'$ is an extension of $\JJ$ such that properties analogous to \eqref{IC1} and \eqref{IC2} hold: There exists a vector $\cc$ such that
\begin{align}
\label{IC12}
\frac{\PP \ee_j}{\| \PP \ee_j \|_2} \cdot \cc &= 1 \quad \forall j \in \JJ', \\
\label{IC13}
\frac{\PP \ee_i}{\| \PP \ee_i \|_2} \cdot \cc &< 1 \quad \forall i \in (\JJ')^c. 
\end{align}

Using the same techniques as in the proof of Theorem \ref{lower_constraint}, we prove the following result.  
\begin{proposition}
Any solution $\yy$ of \eqref{IC3} satisfies   
\[
\mathrm{supp} (\yy) \subseteq \JJ'
\]
for any index set $\JJ' \supseteq \JJ = \textnormal{supp}(\xx^*)$ satisfying \eqref{IC12}-\eqref{IC13}. 
Consequently, a solution of \eqref{IC3} must obey 
\[
\mathrm{supp} (\yy) \subseteq \overline{\JJ}, 
\]
where 
\[
\overline{\JJ} = \bigcap \{\JJ' \; | \; \JJ' \supseteq \JJ \mbox{ and } \exists \, \cc \mbox{ such that } \JJ' \mbox{ satisfies \eqref{IC12}-\eqref{IC13}} \},  
\]
and we assume that $\overline{\JJ} \neq \emptyset$. This holds for any $s \in [\| \xx^* \|_\infty, \infty]$. 
\end{proposition}
\begin{proof}
Let $\yy$ denote a solution of \eqref{IC3}. Very similar to the reasoning leading to \eqref{IC4}, we find that 
\begin{equation*} 
\sum_{(\JJ')^c \cup \JJ'} y_i \| \PP  \ee_i \|_2 \frac{\PP \ee_i}{\| \PP  \ee_i \|_2} \cdot \cc  
= \sum_{\JJ} x_j^* \| \PP  \ee_j \|_2. 
\end{equation*}
Assuming that there exists $i' \in (\JJ')^c$ such that $y_i > 0$, we can argue as we did in 
\eqref{IC4.1}-\eqref{IC5} to conclude that 
\[
\| \WW \yy \|_1 > \| \WW \xx^* \|_1,
\]
where $\xx^* = \sum_\JJ x_j^* \ee_j$. Since $\xx^*$ satisfies the constraints in \eqref{IC3}, there can not exist such $i' \in (\JJ')^c$. It follows that $\mathrm{supp} (\yy) \subseteq \JJ'$. 

\end{proof}

As far as we understand, this result is primarily of theoretical interest because in practice it seems difficult to compute the smallest extension $\overline{\JJ}$. We also note that, even though there exists, by assumption, a vector $\cc$ for each $\JJ'$ such that \eqref{IC12} and \eqref{IC13} hold, it does not follow from our argument that $\overline{\JJ}$ satisfies these conditions. 


\subsection{Optimality conditions} \label{sec:optimality_conditions}
In the previous subsection we showed that the support of any solution of \eqref{IC3} is a subset of the support of the true source $\xx^*$ given the assumptions \eqref{IC1}-\eqref{IC2}. We will now show that these assumptions, which we motivated from a geometrical point view, in fact correspond to the existence of a KKT/Lagrange multiplier for the optimality system associated with \eqref{IC3}. 

First, since $\AAA\xx = \AAA\xx^* \Leftrightarrow \PP\xx = \PP\xx^*$, recall that $\PP = \AAA^\dagger \AAA$, we can formulate the standard KKT-conditions for \eqref{IC3}, see, e.g., \cite[Theorem 3.34]{rus06}, with $s=\infty$, as
\begin{eqnarray*}
  \WW\partial\|\xx\|_1 - \PP\boldsymbol{\lambda} - \boldsymbol{\mu} &\ni& \mathbf{0}, \\
  \PP\xx - \PP\xx^* &=& \mathbf{0}, \\
  \xx \cdot \boldsymbol{\mu} &=& 0, \\ 
  \xx, \boldsymbol{\mu} &\geq& \mathbf{0},
\end{eqnarray*}
where $\partial$ denotes the subgradient, and we have used the chain rule and the fact that the diagonal entries of $\WW$ are positive: 
\begin{equation*}
    \partial_\xx \| \WW \xx\|_1 = \WW\partial\|\WW \xx\|_1 = \WW\partial\|\xx\|_1.
\end{equation*}
Recall that $\xx^* = \sum_\JJ x_j^* \ee_j$, $\xx^* \geq \mathbf{0}$. 
Thus, for $\xx^*$ to be a solution of \eqref{IC3} it follows that there must exist some $\boldsymbol{\lambda}^*$ and $\boldsymbol{\mu}^*$ such that 
\begin{equation} \label{our_KKT}
    [\WW\partial\|\xx^*\|_1]_i \ni 
    \begin{cases}
        [\PP\boldsymbol{\lambda}^*]_i, \quad &i \in \mathcal{J}, \\
        [\PP\boldsymbol{\lambda}^*]_i + \mu_i^*, \quad &i \in \mathcal{J}^c. 
    \end{cases}
\end{equation}
(Where $[\mathbf{v}]_i$ represents the $i$'th component of the Euclidean vector $\mathbf{v}$). 
Here we have used the implication 
\begin{equation*}
\left.
\begin{array}{l}
    \xx^* \cdot \boldsymbol{\mu}^* = 0   \\
    \xx^*, \boldsymbol{\mu}^* \geq \mathbf{0} \\
    \JJ = \mathrm{supp}(\xx^*)
\end{array}
\right\}
      \Longrightarrow \; \; \mu_i^* = 0 \; \; \forall i \in \JJ.
\end{equation*}

If we evaluate the subgradient
\begin{equation*}
    \left[ \partial\|\xx^*\|_1 \right]_i =  
        \begin{cases}
            \{1\}, \quad &i \in \JJ, \\
            [-1,1], \quad &i \in \JJ^c,
        \end{cases}
\end{equation*}
in \eqref{our_KKT}, and combine it with 
\begin{equation*}
    [\PP\boldsymbol{\lambda}^*]_i = \PP\boldsymbol{\lambda^*} \cdot \ee_i = \boldsymbol{\lambda}^* \cdot \PP\ee_i,
\end{equation*}
we derive that
\begin{eqnarray*}
    \|\PP\ee_j\|_2 &=& \PP\ee_j \cdot \boldsymbol{\lambda}^*, \quad j \in \JJ, \\
    \|\PP\ee_j\|_2 &\geq& \PP\ee_j \cdot \boldsymbol{\lambda}^*, \quad j \in \JJ^c,
\end{eqnarray*}
because $\WW = \textnormal{diag} \left( \|\PP\ee_1\|_2, \, \|\PP\ee_2\|_2, \, \ldots, \, \|\PP\ee_n\|_2 \right)$, see \eqref{eq:4}, and due to the fact that $\boldsymbol{\mu}^* \geq \mathbf{0}$. 

If we have strict complementary slackness, i.e., $\mu^*_i > 0$ whenever $x^*_i = 0$, we get the stronger result
\begin{eqnarray*}
    \|\PP\ee_j\|_2 &=& \PP\ee_j \cdot \boldsymbol{\lambda}^*, \quad j \in \JJ, \\
    \|\PP\ee_j\|_2 &>& \PP\ee_j \cdot \boldsymbol{\lambda}^*, \quad j \in \JJ^c,
\end{eqnarray*}
which are assumptions \eqref{IC1}-\eqref{IC2} with $\mathbf{c} = \boldsymbol{\lambda}^*$.

\subsection{Regularized problems} \label{subsec:regularized_problems}
Since $\PP=\AAA^\dagger \AAA$, see \eqref{eq:3}, the null spaces of $\PP$ and $\AAA$ coincide,  
and we may write the basis pursuit problem \eqref{eq5} in the form 
\begin{equation} \label{ICR0.9} 
      \min_{0 \leq \xx \leq s} \| \WW \xx \|_1 \quad \textnormal{subject to} \quad \PP \xx = \AAA^\dagger \bb. 
\end{equation}
We also recall that the entries of the regularization matrix $\WW$ \eqref{eq:4} are defined in terms of the projection $\PP$. Based on these observations, and due to mathematical convenience, we will in this subsection analyse a regularized version of \eqref{ICR0.9} instead of \eqref{eq:3.1}: 
\begin{equation} \label{ICR0.91}
    \min_{0 \leq \xx \leq s} \left\{ \frac{1}{2}\|\PP \xx - \AAA^\dagger \bb\|_2^2 + \alpha \sum_i w_i |x_i| \right\}. 
\end{equation}
More precisely, if $\bb = \bb^\dagger = \AAA \xx^*$, to what extent can we recover $\xx^*$ by solving \eqref{ICR0.91}? Note that, in this case $\AAA^\dagger \bb^\dagger = \AAA^\dagger \AAA \xx^* = \PP \xx^*$. 

We remark that we have {\em not} managed to prove results similar to those presented in this subsection for the problem \eqref{eq:3.1}. In fact, numerical experiments indicate that we do not necessarily obtain as good reconstructions with \eqref{eq:3.1} as with \eqref{ICR0.91} when $\alpha >0$ is not very close to zero.

We now prove a regularized counterpart to Theorem \ref{lower_constraint}. 
\begin{theorem} \label{lower_constraint_ineq}
Let $\xx^* = \sum_\JJ x_j^* \ee_j \geq 0$, $\JJ=\mathrm{supp}(\xx^*)$, and assume that there exist a vector $\dd$ and a constant $\gamma < 1$ such that 
\begin{align}
\label{ICR1}
\frac{\PP \ee_j}{\| \PP \ee_j \|_2} \cdot \dd &= 1 \quad \forall j \in \JJ, \\
\label{ICR2}
\frac{\PP \ee_i}{\| \PP \ee_i \|_2} \cdot \dd &< \gamma \quad \forall i \in \JJ^c. 
\end{align}
Then any  
\begin{equation}
\label{ICR3}
    \yy \in \argmin_{0 \leq \xx \leq s} \underbrace{\left\{\frac{1}{2}\left\|\PP\xx-\PP \overbrace{\sum_\JJ x_j^* \ee_j}^{= \xx^*}\right\|_2^2 + \alpha\|\WW\xx\|_1 \right\}}_{:=\mathcal{T}_\alpha(\xx)}  
\end{equation}
satisfies 
\begin{align} \label{ICR3.1}
    \left| \| \WW \yy \|_1 - \| \WW \xx^*\|_1 \right| &\leq c \, \frac{\sqrt{\alpha}}{1-\gamma}  
\end{align}
and 
\begin{align} \label{ICR3.2}
    \| \WW \yy \|_{1,\JJ^c} := \sum_{\JJ^c} y_i \| \PP \ee_i \|_2 &\leq c \, \frac{\sqrt{\alpha}}{1-\gamma} , 
\end{align}
provided that $s \in [\| \xx^* \|_\infty, \infty ]$. Here, $c$ is a constant which is independent of $\alpha$, $\gamma$ and $s$, but depends on $\xx^*$ and $\dd$. 
\end{theorem}
\subsubsection*{Remark} 
If there exists a vector $\cc$ such that \eqref{IC1} and \eqref{IC2} hold, then there is a vector $\dd$ and constant $\gamma < 1$ such that \eqref{ICR1} and \eqref{ICR2} are satisfied because we consider finite dimensional problems: $\dd = \cc$ and choose $\gamma$ such that  
\begin{align*}
    &\max_{i \in \JJ^c} \left\{ \frac{\PP \ee_i}{\| \PP \ee_i \|_2} \cdot \cc \right\} < \gamma < 1.  
\end{align*}

\begin{proof}[Proof of Theorem \ref{lower_constraint_ineq}]
Assume that $\yy$ belongs to the $\argmin$ set defined in \eqref{ICR3}. Then, because $s \geq \| \xx^* \|_\infty$,
\[
\mathcal{T}_\alpha(\yy) \leq \mathcal{T}_\alpha(\xx^*) = \alpha \| \WW \xx^* \|_1, 
\]
which implies that 
\begin{align}
\label{ICR4}
\alpha \| \WW \yy \|_1 &\leq \alpha \| \WW \xx^* \|_1, \\
\label{ICR5}
    \| \ttau \|_2^2 &\leq 2 \alpha \| \WW \xx^* \|_1,  
\end{align}
where 
\[
\ttau = \PP \yy - \PP \xx^*. 
\]

Now, 
\[
\PP \yy \cdot \dd = \PP \xx^* \cdot \dd + \ttau \cdot \dd, 
\]
which leads to  
\[
\sum_{\JJ \cup \JJ^c} y_i \PP \ee_i \cdot \dd = \sum_\JJ x_j^* \PP \ee_j \cdot \dd + \ttau \cdot \dd. 
\]
Combining assumptions \eqref{ICR1} and \eqref{ICR2} with the constraint $y_i \geq 0$ yield that 
\begin{equation*}
    \gamma \sum_{\JJ^c} y_i \| \PP \ee_i \|_2 + \sum_\JJ y_i \| \PP \ee_i \|_2 \geq \sum_\JJ x_j^* \| \PP \ee_j \|_2 + \ttau \cdot \dd
\end{equation*}
or, also using the fact that $\| \WW \xx^* \|_1 = \sum_\JJ x_j^* \| \PP \ee_j \|_2$ and \eqref{ICR4}, 
\begin{align*}
    \sum_\JJ y_i \| \PP \ee_i \|_2 &\geq \| \WW \xx^* \|_1 - \gamma \sum_{\JJ^c} y_i \| \PP \ee_i \|_2 + \ttau \cdot \dd \\
    &\geq \| \WW \xx^* \|_1 - \gamma \left[ \| \WW \xx^* \|_1 - \sum_\JJ y_i \| \PP \ee_i \|_2 \right] + \ttau \cdot \dd \\
    &= (1-\gamma) \| \WW \xx^* \|_1 + \gamma \sum_\JJ y_i \| \PP \ee_i \|_2  + \ttau \cdot \dd,  
\end{align*}
which, since $\gamma < 1$, implies that 
\begin{equation} \label{ICR6}
    \sum_\JJ y_i \| \PP \ee_i \|_2 \geq \| \WW \xx^* \|_1 + \frac{\ttau \cdot \dd}{1 - \gamma}. 
\end{equation}

From \eqref{ICR4}, keeping in mind that $y_i \geq 0$ and invoking \eqref{ICR6}, it follows that  
\begin{align}
\nonumber
    \| \WW \xx^* \|_1 \geq \| \WW \yy \|_1 &= \sum_{\JJ \cup \JJ^c} y_i \| \PP \ee_i \|_2 \\
    \nonumber
    &\geq \sum_{\JJ} y_i \| \PP \ee_i \|_2 \\
    \label{ICR7}
    &\geq \| \WW \xx^* \|_1 + \frac{\ttau \cdot \dd}{1 - \gamma}, 
\end{align}
and therefore 
\begin{align*}
    \left| \| \WW \yy \|_1 - \| \WW \xx^*\|_1 \right| &\leq \left| \frac{\ttau \cdot \dd}{1 - \gamma} \right| \\
    & \leq \underbrace{\sqrt{2} \| \dd \|_2 \sqrt{\| \WW \xx^*\|_1}}_{:=c} \, \frac{\sqrt{\alpha}}{1 - \gamma} ,
\end{align*}
where we have used the Cauchy-Schwarz inequality and \eqref{ICR5}. This completes the proof of \eqref{ICR3.1}. 

Combing \eqref{ICR4} and \eqref{ICR6} it follows that 
\begin{align*}
    \| \WW \xx^* \|_1 \geq \| \WW \yy \|_1 &= \sum_{\JJ \cup \JJ^c} y_i \| \PP \ee_i \|_2 \\
    &\geq \| \WW \xx^* \|_1 + \frac{\ttau \cdot \dd}{1 - \gamma} + \sum_{\JJ^c} y_i \| \PP \ee_i \|_2, 
\end{align*}
which implies that 
\begin{equation*}
    - \frac{\ttau \cdot \dd}{1 - \gamma} \geq \sum_{\JJ^c} y_i \| \PP \ee_i \|_2 \geq 0,  
\end{equation*}
where we have also used the constraint $y_i \geq 0$. Hence, 
\begin{align*}
    \sum_{\JJ^c} y_i \| \PP \ee_i \|_2 &\leq \left| \frac{\ttau \cdot \dd}{1 - \gamma} \right| \\
    & \leq \sqrt{2} \| \dd \|_2 \sqrt{\| \WW \xx^*\|_1} \, \frac{\sqrt{\alpha}}{1 - \gamma} ,
\end{align*}
which completes the proof of \eqref{ICR3.2}. (Note also that \eqref{ICR7} implies that $\frac{\ttau \cdot \dd}{1 - \gamma} \leq 0$.)
\end{proof}

The next result contains the regularized version of Corollary \ref{strength}. That is, we consider the possibility for determining the strength $k$ of a source with constant strength: 
\begin{equation*}
    \xx^* = k \sum_\JJ \ee_j. 
\end{equation*}
\begin{corollary}
If assumptions \eqref{ICR1} and \eqref{ICR2} hold with $\JJ = \textnormal{supp} (\xx^*)$, then any 
\begin{equation} \label{ICR7.9}
    \yy \in \argmin_{0 \leq \xx \leq s} \underbrace{\left\{\frac{1}{2}\left\|\PP\xx-\PP \overbrace{k \sum_\JJ \ee_j}^{= \xx^*}\right\|_2^2 + \alpha\|\WW\xx\|_1 \right\}}_{=\mathcal{T}_\alpha(\xx)}  
\end{equation}
satisfies 
\begin{align}
    \label{ICR8}
    &\| \WW \yy \|_1 \geq \left( 1 + \frac{(1-\gamma)(k-s)}{k \gamma} \right) \| \WW \xx^*\|_1 - C \, \frac{1}{\gamma} \sqrt{s} \sqrt{\alpha} \quad \mbox{for } s \in [0, k) \cap S, \\
     %
     \label{ICR9}
    &\left| \| \WW \yy \|_1 - \| \WW \xx^*\|_1 \right| \leq c \, \frac{\sqrt{\alpha}}{1-\gamma}  \quad \mbox{for } s \geq k, 
\end{align}
where the set $S$ is defined in \eqref{defS} and $c$ and $C$ are constants which are independent of $\gamma$, $s$ and $\alpha$. 
\end{corollary}
\begin{proof}
Inequality \eqref{ICR9} is the same as inequality \eqref{ICR3.1}, and we will thus focus on \eqref{ICR8}. Let $$s \in [0, k) \cap S$$ and assume that $\yy$ belongs to the $\argmin$ set in \eqref{ICR7.9}. Then, according to the definition \eqref{defS} of $S$, there exists $\zz$, with $\| \zz \|_{\infty} \leq s$, such that $\AAA \zz= \AAA \xx^*$, which implies that $\PP \zz = \PP \xx^*$. It follows that 
\[
\mathcal{T}_\alpha(\yy) \leq \mathcal{T}_\alpha(\zz) = \alpha \| \WW \zz \|_1 \leq \alpha s \| \WW \mathbf{1} \|_1 , 
\]
which implies that 
\begin{align}
\label{ICR10}
\alpha \| \WW \yy \|_1 &\leq \alpha s \| \WW \mathbf{1} \|_1, \\
\label{ICR11}
    \| \ttau \|_2^2 &\leq 2 \alpha s \| \WW \mathbf{1} \|_1,  
\end{align}
where 
\[
\ttau = \PP \yy - \PP \xx^* 
\]
and $\mathbf{1}$ denotes the vector with all components equal to $1$. 

Next, 
\[
\PP \yy \cdot \dd = \PP \xx^* \cdot \dd + \ttau \cdot \dd
\]
and it follows that   
\[
\sum_{\JJ \cup \JJ^c} y_i \PP \ee_i \cdot \dd = k \sum_\JJ \PP \ee_j \cdot \dd + \ttau \cdot \dd, 
\]
which we combine with assumptions \eqref{ICR1} and \eqref{ICR2} and the constraint $y_i \geq 0$ to find that 
\begin{equation} \label{ICR12}
    \gamma \sum_{\JJ^c} y_i \| \PP \ee_i \|_2 + \sum_\JJ y_i \| \PP \ee_i \|_2 \geq k \sum_\JJ \| \PP \ee_j \|_2 + \ttau \cdot \dd. 
\end{equation}
It thus follows that 
\begin{equation*}
     \sum_{\JJ \cup \JJ^c} y_i \| \PP \ee_i \|_2 \geq (1-\gamma) \sum_{\JJ^c} y_i \| \PP \ee_i \|_2 + k \sum_\JJ \| \PP \ee_j \|_2 + \ttau \cdot \dd, 
\end{equation*}
or 
\begin{equation} \label{ICR13}
     \| \WW \yy \|_1 \geq (1-\gamma) \sum_{\JJ^c} y_i \| \PP \ee_i \|_2 + \| \WW \xx^* \|_1 + \ttau \cdot \dd
\end{equation}
because $\| \WW \yy \|_1 = \sum_{\JJ \cup \JJ^c} y_i \| \PP \ee_i \|_2$ and $\| \WW \xx^* \|_1 = k \sum_\JJ \| \PP \ee_j \|_2$. 

Employing the constraint $0 \leq \yy \leq s$ and \eqref{ICR12} we find that 
\begin{align*}
    \gamma \sum_{\JJ^c} y_i \| \PP \ee_i \|_2 &\geq (k-s) \sum_\JJ \| \PP \ee_j \|_2 + \ttau \cdot \dd \\
    &= \frac{k-s}{k} \| \WW \xx^* \|_1 + \ttau \cdot \dd, \\
\end{align*}
which combined with \eqref{ICR13} yields 
\begin{align*}
    \| \WW \yy \|_1 &\geq (1-\gamma) \frac{(k-s)}{k \gamma} \| \WW \xx^* \|_1  + \| \WW \xx^* \|_1 + \frac{1}{\gamma}\ttau \cdot \dd \\
    &\geq \left( 1+ \frac{(1-\gamma)(k-s)}{k \gamma} \right) \| \WW \xx^* \|_1 - \frac{1}{\gamma} \sqrt{\alpha} \sqrt{s} \underbrace{\sqrt{2} \sqrt{\| \WW \mathbf{1} \|_1} \| \dd \|_2}_{=C},  
\end{align*}
where we have also used \eqref{ICR11}. 
This finishes the proof. 
\end{proof}

\section{Numerical experiments}
\label{sec:numerical_experiments}
We discretized the boundary value problem \eqref{eq2} with the finite element method, using first-order Lagrange elements, which yields the stiffness matrix $\mathsf{L}_\epsilon$. The forward matrix associated with \eqref{eq1}-\eqref{eq2} then becomes 
$$\AAA = \mathsf{M}_\partial^{1/2}\mathsf{L}_\epsilon^{-1}\mathsf{M},$$
where $\mathsf{M}$ and $\mathsf{M}_\partial$ denote the standard and "boundary" mass matrices, respectively. With this approach, the local "hat" finite element functions $\{ \psi_i \}$ are directly associated with the standard unit basis vectors $\{ \ee_i \}$.

Motivated by the discussion presented at the beginning of subsection \ref{subsec:regularized_problems}, we performed numerical experiments with 
\begin{equation} \label{eq:num1}
    \min_{0 \leq \xx \leq s} \left\{ \frac{1}{2}\|\PP \xx - \AAA^\dagger \bb\|_2^2 + \alpha \sum_i w_i |x_i| \right\}
\end{equation}
instead of using \eqref{eq:3.1}. Here, the matrices $\PP = \AAA^\dagger \AAA$ and $\WW$ are defined in equations \eqref{eq:3} and \eqref{eq:4}, respectively, and the size of the regularization parameter $\alpha$ is specified in the captions of the figures below. As is mentioned at the end of subsection \ref{subsec:weight_matrix}, truncated SVD, using $20$ singular values, was employed to obtain a computationally reliable approximation of $\AAA^\dagger$. 

We numerically tested whether one can recover $\xx^*$ by solving \eqref{eq:num1} when $\bb = \bb^\dagger = \AAA \xx^*$.  
If not stated otherwise, the true source always had constant unit strength: 
\begin{equation*}
    \xx^* = \sum_\JJ \ee_j.
\end{equation*} 

Since $\PP$ has a nontrivial null space and the regularization term in \eqref{eq:num1} is not strictly convex, \eqref{eq:num1} can potentially have several solutions. This may be "handled" by adding a small amount of quadratic regularization. However, we did not do that, and instead we present the outcome of applying the Alternating Direction Method of Multipliers (ADMM) to \eqref{eq:num1}: Using the zero initial guess and performing $5000$ iterations. We employed Python, in combination with the FEniCS software tool, to realize the numerical experiments.    

According to corollaries \ref{cor:injective} and \ref{lower_upper_constraint}, solving the basis pursuit problem \eqref{IC3} can, under certain circumstances, lead to \textit{exact} recovery of the true source(s), even though $\AAA$ has a large null space. In order to illuminate these results, we committed an "inverse crime" in our two first examples: The same coarse $17 \times 17$ grid was used to represent the source(s) in both the forward and inverse simulations, and the state $u$ was discretized in terms of a $33 \times 33$ mesh. With this setup, the true source can be exactly represented on the grid used in the inverse computations. This leads to unrealistically accurate recoveries in the sense that, in addition to finding the position, size and magnitude, also the exact shape of the source(s) is recovered -- as we will see in examples 1 and 2 below. On the other had, let us emphasize that, even when inverse crimes are committed, the recovery of $\xx^*$ is not trivial because $\AAA$ has a large null space: In example 1 we show that unweighted sparsity regularization fails to provide adequate results.   

In examples 3-5 we avoided inverse crimes: The synthetic boundary data $\bb$ was generated using a grid with $97 \times 97$ nodes for both the source and the state, whereas a $49 \times 49$ grid was used for  these quantities in the inverse solution procedure. 

The unit square $\Omega = (0,1) \times (0,1)$ was used as domain, and in all the cases we ran experiments with both $\epsilon = -1$ and $\epsilon = 1$ in \eqref{eq2}, i.e., with both Helmholtz' equation and the screened Poisson equation. The results obtained with these two state equations were visually indistinguishable, and we therefore, in each example, only present results obtained for one of these model problems: In example 1 and examples 2-5 we present results for $\epsilon=1$ and $\epsilon=-1$, respectively.  

\subsection*{Example 1: Point sources}
In the first example we attempted to locate five single/point sources, see panel (a) in figure \ref{fig:ex1}. Recall that, if assumptions \eqref{IC1}-\eqref{IC2} hold, then Theorem \ref{lower_constraint} guarantees that the solution of the basis pursuit problem \eqref{IC3} has support contained in $\JJ = \textnormal{supp} (\xx^*)$ -- the indices associated with the true point sources. Panel (b) in figure \ref{fig:ex1} shows the solution of \eqref{eq:num1} computed with no upper bound, i.e., $s = \infty$. We observe that all the five point sources, with unit strength, are recovered. This suggests that $\AAA$ is injective on $\{\zz \ | \ \textnormal{supp}(\zz) \subset \JJ \textnormal{ and } 0 \leq \zz\}$, see Corollary \ref{cor:injective}. Indeed, figure \ref{fig:ex1mag} shows that also sources with varying magnitude are perfectly recovered. 

\begin{figure}[H]
    \centering
    \begin{subfigure}[b]{0.45\linewidth}        
        \centering
        \includegraphics[width=\linewidth]{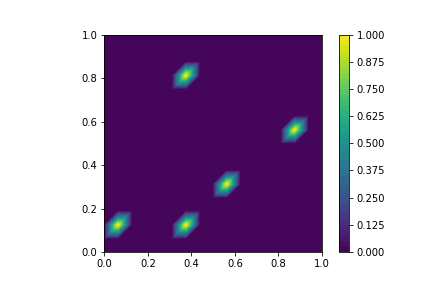}
        \caption{True sources. \break}
    \end{subfigure}\par
    \begin{subfigure}[b]{0.45\linewidth}        
        \centering
        \includegraphics[width=\linewidth]{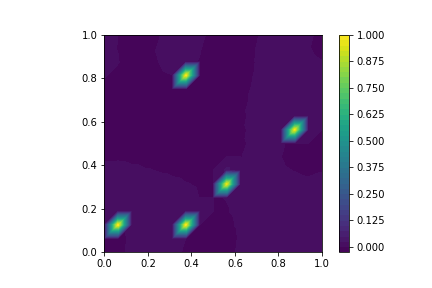}
        \caption{Inverse solution computed with the constraint $0 \leq \xx$.}
    \end{subfigure}\par
    \begin{subfigure}[b]{0.45\linewidth}        
        \centering
        \includegraphics[width=\linewidth]{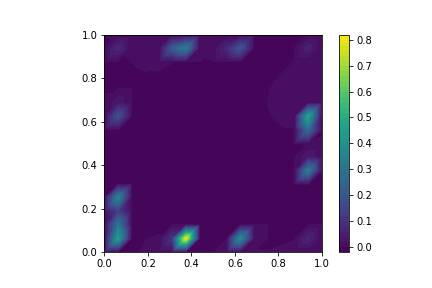}
        \caption{Inverse solution computed with the constraint $0 \leq \xx$ and $\WW = \mathsf{I}$, i.e., unweighted regularization}
    \end{subfigure}    
    \caption{Example 1. Comparison of the true point sources and the inverse solutions computed with weighted and unweighted regularization. The regularization parameter was $\alpha = 10^{-4}$.}
    \label{fig:ex1}
\end{figure}

It was mentioned in the introduction to this section that we committed an inverse crime in this example. However, we emphasis that the the recovery of the present point sources is not trivial, which we illustrate by displaying the solution of 
\begin{equation*} 
    \min_{0 \leq \xx} \left\{ \frac{1}{2}\|\PP \xx - \AAA^\dagger \bb\|_2^2 + \alpha \sum_i |x_i| \right\}
\end{equation*}
in panel (c) in figure \ref{fig:ex1}. That is, we ran the same test without weights in the regularization term. We observe that this procedure fails to recover the true point sources. Also, since the range of the function plotted in panel (c) is a subset of $[0,1]$, it would not help to add an upper constraint, i.e, to employ the box constraint $0 \leq \xx \leq 1$, where we recall that the strength of the true sources is $1$. 


\begin{figure}[h]
    \centering
    \begin{subfigure}[b]{0.45\linewidth}        
        \centering
        \includegraphics[width=\linewidth]{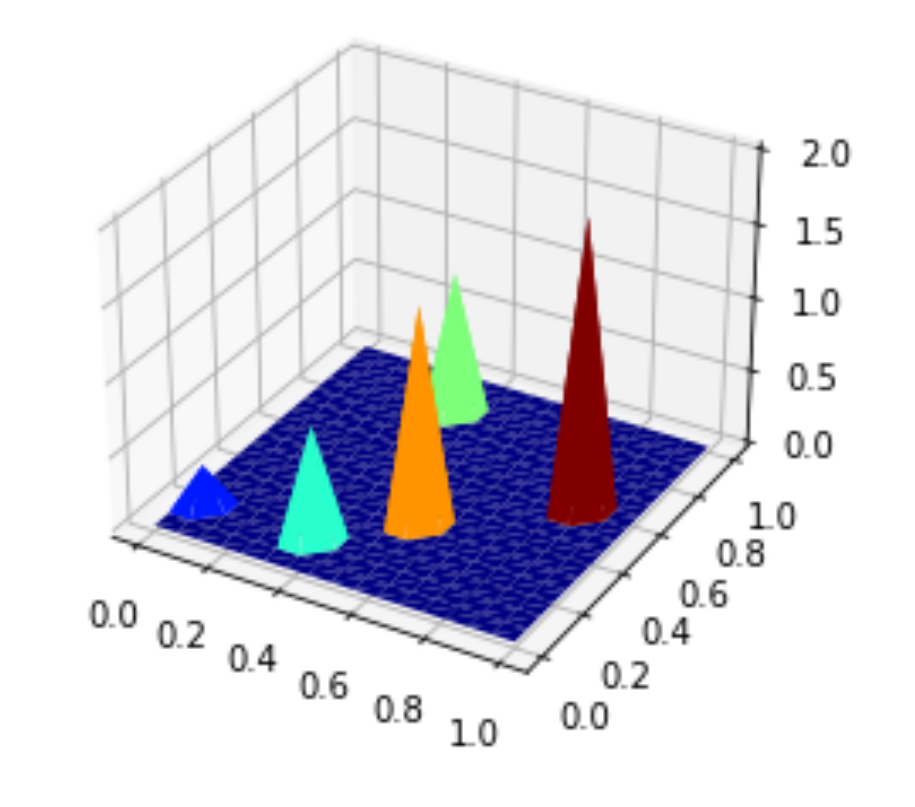}
        \caption{True sources. \break}
    \end{subfigure}
    \begin{subfigure}[b]{0.45\linewidth}        
        \centering
        \includegraphics[width=\linewidth]{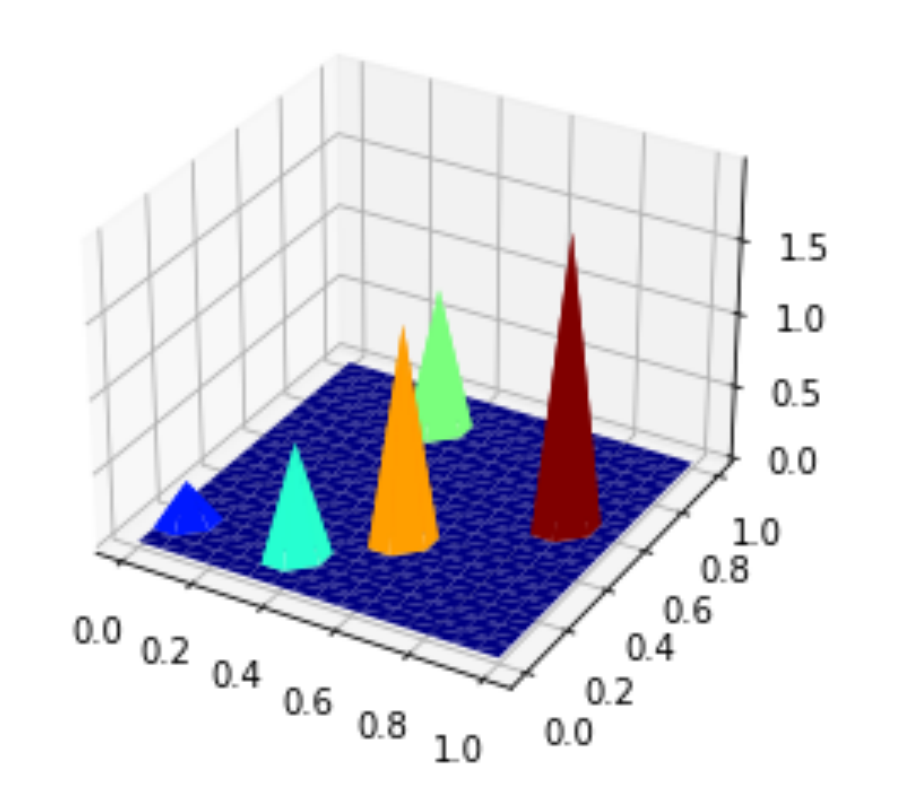}
        \caption{Inverse solution computed with the constraint $0 \leq \xx$.}
    \end{subfigure}
    \caption{Example 1. Comparison of the true point sources with varying magnitude and the inverse solution. The regularization parameter was $\alpha = 10^{-4}$. Note that the same grid was used in the forward and inverse simulations.}
    \label{fig:ex1mag}
\end{figure}

\subsection*{Example 2: Three rectangular sources}
Our second example concerns the recovery of three composite sources. More specifically, we have three separate rectangles, each consisting of several cells, see panel (a) in figure \ref{fig:ex2}. 
The right panel of this figure shows that, when only using the lower constraint $\xx \geq 0$, our method is successful in localizing the three sources, although with underestimated support and overestimated magnitudes. This is in accordance with Theorem \ref{lower_constraint}: The support of the inverse solution is contained in the support of the true sources. However, if the matrix $\AAA$ is not injective on the set $\{\zz \; | \; \textnormal{supp}(\zz) \subset \JJ \textnormal{ and } 0 \leq \zz \leq s\}$, we can not guarantee exact/perfect recovery. For $s = \infty$, the matrix $\AAA$ does not appear to be injective on this set. 

To enforce injectivity, we can set the upper box constraint value $s$ equal to the strength $k$ of the true sources, see Corollary \ref{lower_upper_constraint}. Nevertheless, in practise $k$ is typically unknown and it must somehow be estimated. We will now address this issue.  
\begin{figure}[H]
    \centering
    \begin{subfigure}[b]{0.45\linewidth}        
        \centering
        \includegraphics[width=\linewidth]{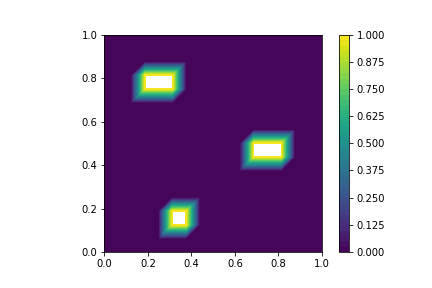}
        \caption{True sources.\break}
    \end{subfigure}
    \begin{subfigure}[b]{0.45\linewidth}        
        \centering
        \includegraphics[width=\linewidth]{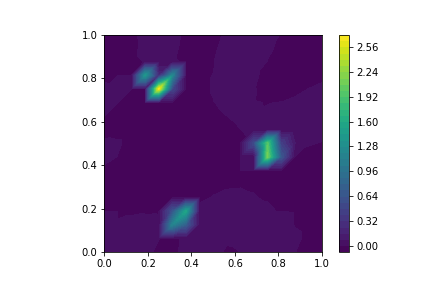}
        \caption{Inverse solution computed with the constraint $0 \leq \xx$.}
    \end{subfigure}\par
    \caption{Example 2. Comparison of the true rectangular sources and the inverse solution. The regularization parameter was $\alpha = 10^{-4}$. Note that the same grid was used in the forward and inverse simulations.}
    \label{fig:ex2}
\end{figure}

\subsubsection*{Determining the strength}
In connection with Corollary \ref{strength} we presented an algorithm for identifying the magnitude $k$ of a constant strength source. Analogously, for the regularized problem \eqref{ICR7.9}, one can use inequalities \eqref{ICR8}-\eqref{ICR9} to approximately determine $k$:

\noindent We can combine \eqref{ICR8}-\eqref{ICR9} to deduce that a solution $\yy$ of \eqref{eq:num1} must satisfy 
\begin{equation*}
    \|\WW\yy\|_1 \sim 
    \begin{cases}
      \left( 1 + \frac{(1-\gamma)(k-s)}{k \gamma} \right) \| \WW \xx^*\|_1, \ &s \in [0, k) \cap S, \\ 
      \|\WW\xx^*\|_1, \ &s \geq k,
    \end{cases}
\end{equation*}
provided that $\alpha > 0$ is rather small. 
Thus, we expect that the $\|\WW\cdot\|_1$-norm of the inverse solution $\yy$ is (relatively) constant for $s > k$. Furthermore, for $s < k$, the size of $\|\WW\cdot\|_1$ should decrease approximately linearly as a function of $s$. Hence, when plotting $\|\WW\yy\|_1$ as a function of the upper bound $s$, we should see an "L-shaped" curve, which is indeed what figure \ref{fig:ex2lcurve} displays. Thus, we select the upper box constraint to be the s-value associated with the vertex of this "L-curve".   

With this strategy we obtained $s=1$ for the present model problem, which is indeed the strength of the true sources. The result is displayed in figure \ref{fig:ex2choices}c). All the three rectangles are exactly identified. 

Figure \ref{fig:ex2choices} also displays the inverse solutions computed with various other values of $s$. In panels (a) and (b), $s > 1 = k$ and, as expected, the support is a subset of $\JJ = \textnormal{supp}(\xx^*)$, see Theorem \ref{lower_constraint}. In panels (d)-(f), $s < 1 = k$, which renders the true source infeasible. For this particular example, an erroneous choice of $s$ is rather "forgiving" in the sense that the support of the inverse solution is either contained inside the support of the true source (when $s$ is chosen too large) or contained in a region surrounding the support of the true source (when $s$ is chosen too small).  
\begin{figure}[h]
    \centering
    \includegraphics[width=0.6\linewidth]{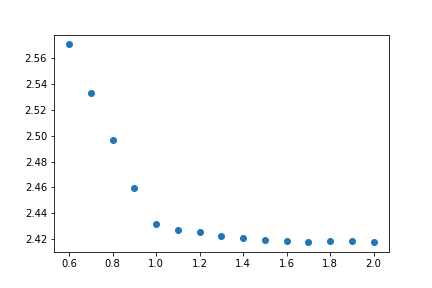}
    \caption{Example 2. Plot of the $\|\WW\cdot\|_1$-norm of the solution $\yy = \yy(s)$ of \eqref{eq:num1} for different values of $s$. Here, $\alpha = 10^{-4}.$}
    \label{fig:ex2lcurve}
\end{figure}

\subsection*{Example 3: The square, the circle and the triangle}
In the remaining examples we avoided inverse crimes: A finer grid was employed in the forward simulations than in the inverse solution procedures. This yields a more realistic setup, but our analysis is not applicable and exact recovery is not possible.

Using this setup, can we still recover the size and position of the sources, but not necessarily the shape? To illuminate the performance of the proposed method under these circumstances, we considered three sources with different shape: a square, a rectangle and a circle, see figure \ref{fig:ex3}(a).

The solution computed with only a lower box constraint $\xx \geq 0$ is shown in panel (b), and we observe that 
the support of the inverse solution is a subset of the support of the true sources. When we also apply the upper box constraint, then the size of the sources are roughly recovered, but the exact geometry is "lost", as seen in panel (c). 

Note that the upper box constraint $s$ equaled $1.2$ in this example. This value was determined manually by the "L-curve" procedure described in the previous example. More precisely, panel (a) in figure \ref{fig:ex3lcurve} shows the $\|\WW\cdot\|_1$-norm of the inverse solution $\yy = \yy(s)$ as a function of $s$, and panel (b) displays the corresponding finite difference approximation of the "derivative" of this curve. Compared with the idealized example considered above, cf. figure \ref{fig:ex2lcurve}, the task of selecting an appropriate value for $s$ is now much more challenging. (Similar observations were made in connection with examples 4 and 5 below.) The value $1.2$, for $s$, was chosen because the dotted-curve in panel (b) in figure \ref{fig:ex3lcurve} has a significant "jump" at this location. 

\begin{figure}[H]
    \centering
    \begin{subfigure}[b]{0.45\linewidth}        
        \centering
        \includegraphics[width=\linewidth]{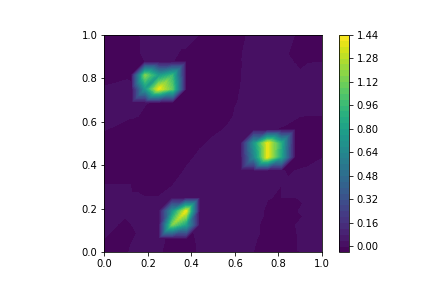}
        \caption{s = 1.4}
    \end{subfigure}
    \begin{subfigure}[b]{0.45\linewidth}        
        \centering
        \includegraphics[width=\linewidth]{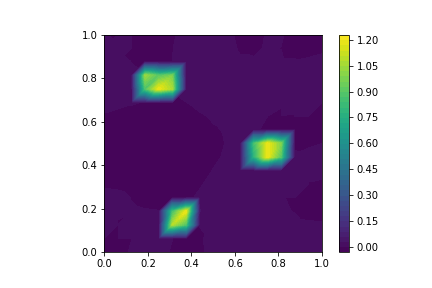}
        \caption{s = 1.2}
    \end{subfigure}\par
    \begin{subfigure}[b]{0.45\linewidth}        
        \centering
        \includegraphics[width=\linewidth]{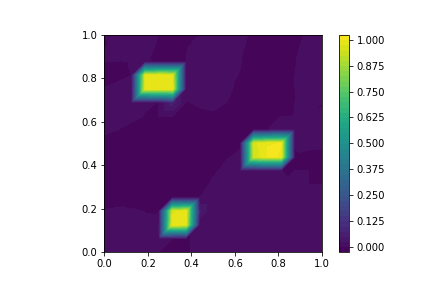}
        \caption{s = 1.0}
    \end{subfigure}
    \begin{subfigure}[b]{0.45\linewidth}        
        \centering
        \includegraphics[width=\linewidth]{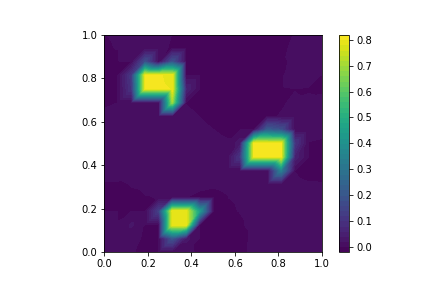}
        \caption{s = 0.8}
    \end{subfigure}\par    
    \begin{subfigure}[b]{0.45\linewidth}        
        \centering
        \includegraphics[width=\linewidth]{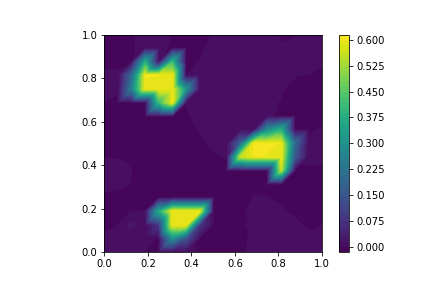}
        \caption{s = 0.6}
    \end{subfigure}
    \begin{subfigure}[b]{0.45\linewidth}        
        \centering
        \includegraphics[width=\linewidth]{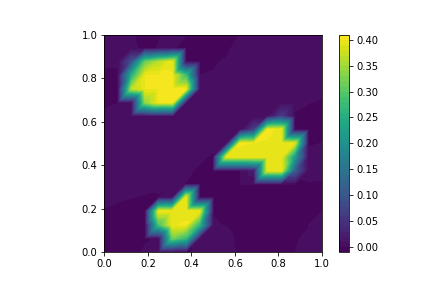}
        \caption{s = 0.4}
    \end{subfigure}
    \caption{Example 2. Inverse solutions computed with different values for the upper box constraint $s$. The regularization parameter was $\alpha = 10^{-4}$, and the true rectangular sources are displayed in figure \ref{fig:ex2}(a). The same grid was used in the forward and inverse simulations.}
    \label{fig:ex2choices}
\end{figure}

\begin{figure}[H]
    \centering
    \begin{subfigure}[b]{0.45\linewidth}        
        \centering
        \includegraphics[width=\linewidth]{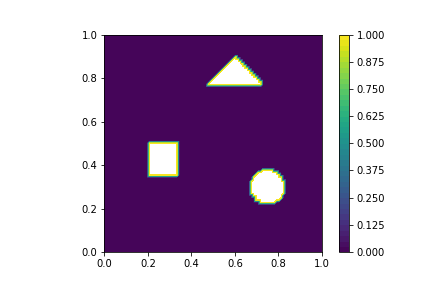}
        \caption{True sources.}
    \end{subfigure}\par
    \begin{subfigure}[b]{0.45\linewidth}        
        \centering
        \includegraphics[width=\linewidth]{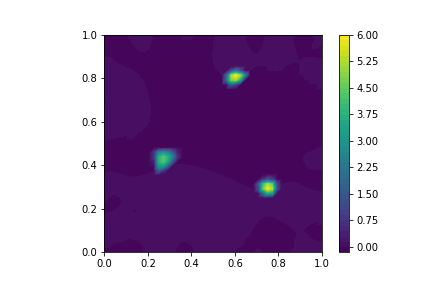}
        \caption{Inverse solution computed with the constraint $0 \leq \xx$.}
    \end{subfigure}\par
    \begin{subfigure}[b]{0.45\linewidth}        
        \centering
        \includegraphics[width=\linewidth]{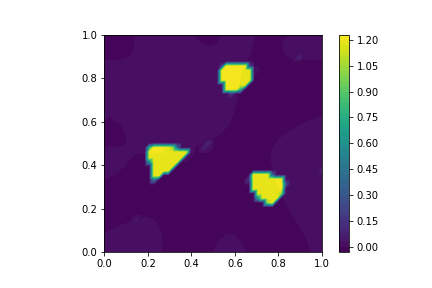}
        \caption{Inverse solution computed with the box constraint $0 \leq \xx \leq 1.2$.}
    \end{subfigure}\par
    \caption{Example 3. Comparison of the true sources and the inverse solutions. The regularization parameter was $\alpha = 10^{-4}$.}
    \label{fig:ex3}
\end{figure}

\begin{figure}[H]
    \centering
    \begin{subfigure}[b]{0.45\linewidth}        
        \centering
        \includegraphics[width=\linewidth]{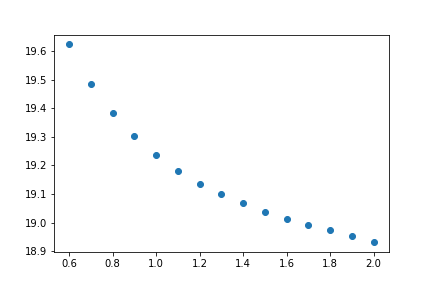}
        \caption{$\|\WW\yy(s)\|_1$}
    \end{subfigure}
    \begin{subfigure}[b]{0.45\linewidth}        
        \centering
        \includegraphics[width=\linewidth]{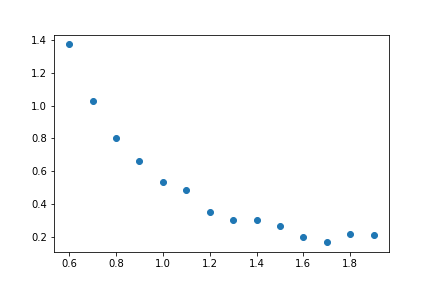}
        \caption{$\frac{\|\WW\yy(s)\|_1 - \|\WW\yy(s')\|_1}{s - s'}$}
    \end{subfigure}\par
    \caption{Example 3. The $\| \WW \cdot \|_1$-norm of the solution $\yy=\yy(s)$ of \eqref{eq:num1}, as a function of $s$, and the associated difference approximation of its "derivative". Here, $\alpha = 10^{-4}.$}
    \label{fig:ex3lcurve}
\end{figure}

\subsection*{Example 4: A horseshoe-shaped source}
In examples 4 and 5 we added noise to the data. More specifically, we computed
\begin{equation*}
   \bb = \AAA\xx^* + \tau\boldsymbol{\rho},
\end{equation*}
where $\tau$ is a scalar and $\boldsymbol{\rho}$ is a vector containing normally distributed numbers with zero mean and standard deviation $1$. We refer to the ratio
\begin{equation*}
    \frac{\tau}{\max\{\bb\} - \min\{\bb\}}
\end{equation*}
as the noise level (in percentage). Furthermore, we selected the regularization parameter $\alpha$ based on Morozov's discrepancy principle.

The first example with noisy data concerns a horseshoe-shaped source. A precise recovery of this geometry appears to be more challenging than the cases considered above because the source now "partially encloses" a piece of the domain, see panel (a) in figure \ref{fig:ex4a}. This suggests that it might be more difficult to determine the boundary of the "inner" part of the horseshoe. 

Panels (b)-(d) in figure \ref{fig:ex4a} show that using only the lower constraint $\xx \geq 0$ results in a highly overestimated magnitude. Nevertheless, without noise, the support of the solution is a subset of the support of the true source. When noise is present, the branches (vertical bars) of the horseshoe is less precisely recovered.  

Figure \ref{fig:ex4b} contains results computed with the box constraint $0 \leq \xx \leq 1$, and we thus used the exact/true upper constraint $s=1$. In panel (b), the case without noise, we get a rather good approximation of the horseshoe. For the problem with 1\% noise, the branches are "bent outwards", whereas they are almost completely "gone" for case with 5\% noise.
\begin{figure}[H]
    \centering
    \begin{subfigure}[b]{0.4\linewidth}        
        \centering
        \includegraphics[width=\linewidth]{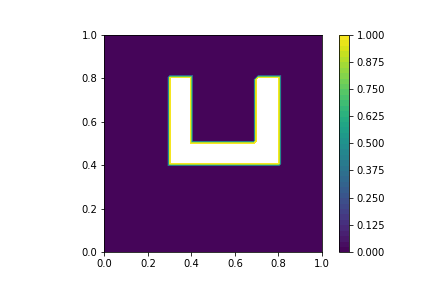}
        \caption{True source.}
    \end{subfigure}
    \begin{subfigure}[b]{0.4\linewidth}        
        \centering
        \includegraphics[width=\linewidth]{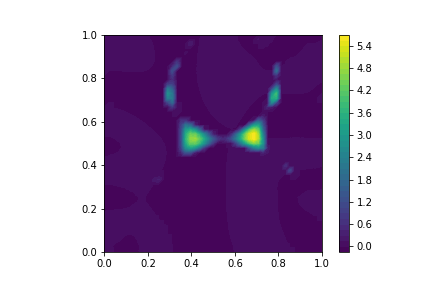}
        \caption{0\% noise ($\alpha = 10^{-4}$).}
    \end{subfigure}\par
    \begin{subfigure}[b]{0.4\linewidth}        
        \centering
        \includegraphics[width=\linewidth]{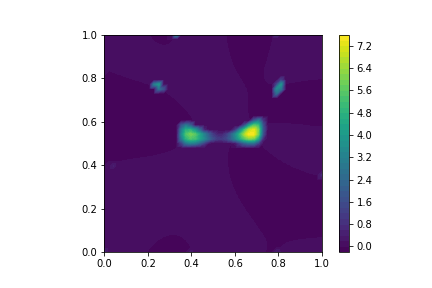}
        \caption{1\% noise ($\alpha = 0.01$).}
    \end{subfigure}
    \begin{subfigure}[b]{0.4\linewidth}        
        \centering
        \includegraphics[width=\linewidth]{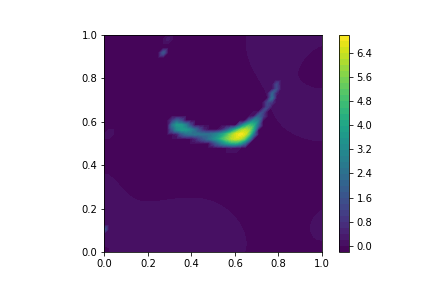}
        \caption{5\% noise ($\alpha = 0.15$).}
    \end{subfigure}    
    \caption{Example 4. Comparison of the true source and the inverse solutions computed with the  constraint $0 \leq \xx$ for different noise levels. }
    \label{fig:ex4a}
\end{figure}

\begin{figure}[H]
    \centering
    \begin{subfigure}[b]{0.4\linewidth}        
        \centering
        \includegraphics[width=\linewidth]{imgs/helmholtz/horseshoe/horseshoeSource.png}
        \caption{True source.}
    \end{subfigure}
    \begin{subfigure}[b]{0.4\linewidth}        
        \centering
        \includegraphics[width=\linewidth]{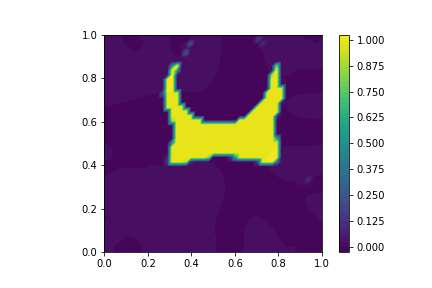}
        \caption{0\% noise ($\alpha = 10^{-4}$).}
    \end{subfigure}\par
    \begin{subfigure}[b]{0.4\linewidth}        
        \centering
        \includegraphics[width=\linewidth]{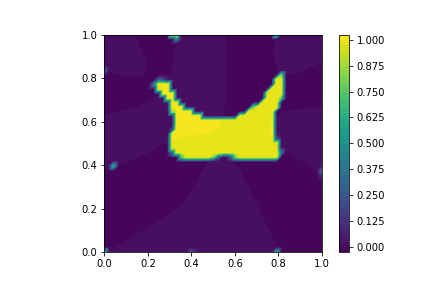}
        \caption{1\% noise ($\alpha = 0.01$).}
    \end{subfigure}
    \begin{subfigure}[b]{0.4\linewidth}        
        \centering
        \includegraphics[width=\linewidth]{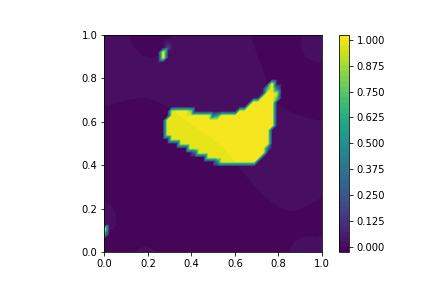}
        \caption{5\% noise ($\alpha = 0.15$).}
    \end{subfigure}    
    \caption{Example 4. Comparison of the true source and the inverse solutions computed with the constraint $0 \leq \xx \leq 1$ for different noise levels. }
    \label{fig:ex4b}
\end{figure}

\subsection*{Example 5: A hollow rectangle}
In the final example we tried to recover a source shaped like a hollow rectangle (picture frame), see figure \ref{fig:ex5}(a). This source was deliberately constructed such that the recovery becomes very challenging. That is, it seems unrealistic to detect the hole/cavity inside the rectangle when only boundary data is available.

Indeed, panels (b)-(d) in figure \ref{fig:ex5} show that our methodology fails to recover the hole inside the rectangle. On the other hand, the algorithm is quite successful in determining the "outer" boundary of the true source for the cases without noise and with 1\% noise. When the data contains 5\% noise, both the position and the size are reasonably approximated, but the shape is too "circular".  
\begin{figure}[h]
    \centering
    \begin{subfigure}[b]{0.4\linewidth}        
        \centering
        \includegraphics[width=\linewidth]{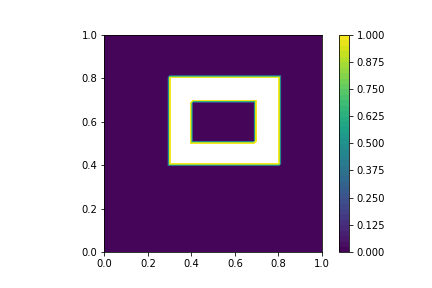}
        \caption{True source.}
    \end{subfigure}
    \begin{subfigure}[b]{0.4\linewidth}        
        \centering
        \includegraphics[width=\linewidth]{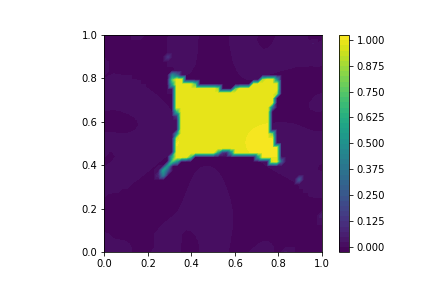}
        \caption{0\% noise ($\alpha = 10^{-4}$).}
    \end{subfigure}\par
    \begin{subfigure}[b]{0.4\linewidth}        
        \centering
        \includegraphics[width=\linewidth]{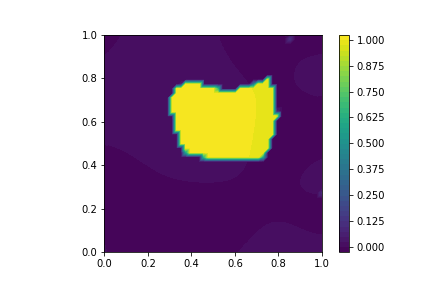}
        \caption{1\% noise ($\alpha = 0.05$).}
    \end{subfigure}
    \begin{subfigure}[b]{0.4\linewidth}        
        \centering
        \includegraphics[width=\linewidth]{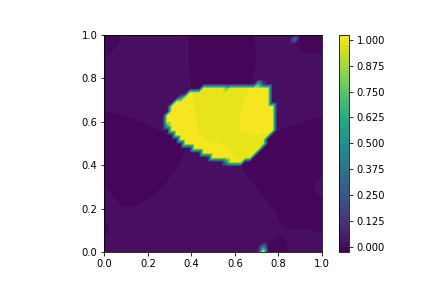}
        \caption{5\% noise ($\alpha = 0.2$).}
    \end{subfigure}    
    \caption{Example 5. Comparison of the true source and the inverse solutions computed with the box constraint $0 \leq \xx \leq 1$ for different noise levels. }
    \label{fig:ex5}
\end{figure}

\section{Concluding remarks}
\label{sec:concluding_remarks}
The main purpose of this paper has been to address the non-uniqueness issue arising in connection with some inverse problems. We have explored this matter for the task of identifying the source term in an elliptic PDE from boundary data. Despite the presence of a large null space, we have shown that the use of weighted sparsity regularization and box constraints can lead to rather accurate recovery. In fact, in an idealized setting, exact reconstruction might be possible. We demonstrated this behavior for two test problems, where we, for the sake of exemplifying the exact recovery, had to commit the "inverse crime" of using identical grids for the forward and inverse computations. Under less ideal conditions, i.e., without "inverse crimes", we still observe that the position and roughly the size of the sources can be computed. That is, the shape can not be perfectly reconstructed -- which is as one would expect. 

For the sake of generality, we developed the methodology and the analysis in terms of Euclidean spaces. Our approach can thus be applied to a broad range of inverse problems, assuming that one seeks sparse solutions. We presented mathematical results for both the basis pursuit formulation of the problem and for its regularized version. It remains to develop an infinite dimensional counterpart to the theory presented in this paper and to test the performance of the methodology on real-world data. 

This investigation was motivated by the fact that standard (unweighted) sparsity regularization do not provide adequate results for the source identification task considered in this paper.

\bibliographystyle{abbrv}
\bibliography{references}

\end{document}